\documentclass[12pt,oneside,english]{amsart}
\usepackage[latin1]{inputenc}
\usepackage{amssymb}

\makeatletter
\newtheorem{theorem}{Theorem}
\newtheorem{lemma}{Lemma}

\newtheorem{corollary}{Corollary}

\usepackage{babel}

\makeatother
\begin{document}
\baselineskip=17pt
\title[Fast computation of density of exponentially sequences]{A fast computation
 of density of exponentially $S$-numbers}

\author{Vladimir Shevelev}
\address{ Department of Mathematics \\Ben-Gurion University of the
 Negev\\Beer-Sheva 84105, Israel. e-mail:shevelev@bgu.ac.il}


\begin{abstract}
The author \cite{4} proved that, for every set $S$ of positive integers containing 1
(finite or infinite) there exists the density $h=h(E(S))$ of the set $E(S)$
of numbers whose prime factorizations contain exponents only from $S,$  and gave an
explicit formula for $h(E(S)).$  In this paper we give an equivalent polynomial
formula for $\log h(E(S))$ which allows to get a fast calculation of $h(E(S)).$
\end{abstract}

\maketitle

\section{Introduction}

Let $\mathbf{G}$ be the set of all finite or infinite increasing sequences
of positive integers beginning with 1. For a sequence $S=\{s(n)\}, n\geq1,$ from
$\mathbf{G},$ a positive number $N$ is called an exponentially
$S$-number $(N\in E(S)),$ if all exponents in its prime power factorization are in
$S.$ The author \cite{4} proved that, for every sequence $S\in \mathbf{G},$
the sequence of exponentially $S$-numbers has a density
$h=h(E(S))\in [\frac{6}{\pi^2}, 1].$
More exactly, the following theorem was proved in \cite{4}:
\begin{theorem}\label{t1}
For every sequence $S\in \mathbf{G}$ the sequence of exponentially $S$-numbers has
a density $h=h(E(S))$ such that
\begin{equation}\label{1}
\sum_{i\leq x,\enskip i\in E(S)} 1 = h(E(S))x+O(\sqrt{x}\log x e^{c\frac{\sqrt{\log x}}
{\log \log x}}),
\end{equation}
with $c=4\sqrt{\frac{2.4}{\log 2}}=7.443083...$ and
\begin{equation}\label{2}
h(E(S))=\prod_{p}\left(1+\sum_{i\geq2}\frac{u(i)-u(i-1)}{p^i}\right),
\end{equation}
where the product is over all primes, $u(n)$ is the characteristic function of
sequence $S:\enskip u(n)=1,$ if\enskip
 $n\in S$ and $u(n)=0$ otherwise.
\end{theorem}
In case when $S$ is the sequence of square-free numbers (see Toth \cite{6}) Arias de
Reyna [5,A262276], using the Wrench method of fast calculation [7], did the
calculation of $h$ with a very high degree of accuracy.
In this paper, using Wrench's method for formula (\ref{2}), we find a general representation of $h(E(S))$ based on a special polynomial over partitions of $n$ which allows to get
a fast calculation of $h(E(S))$ for every $S\in \mathbf{G}.$
Note also that Wrench's method was successfully realized in a special
case by Arias de Reyna, Brent and van de Lune in \cite{2}.\newline
\indent Everywhere below we write $\{h(E(S))\},$ understanding
$\{h(E(S))\}|_{S\in \mathbf{G}}.$
\section{A computing idea in Wrench's style}
Consider function given by power series
\begin{equation}\label{3}
F_S(x)=1+\sum_{i\geq2}(u(i)-u(i-1))x^i,\enskip x\in (0,\frac{1}{2}].
\end{equation}
Since $u(n)-u(n-1)\geq-1,$ then $F_S(x)\geq1-\frac{x^2}{1-x}>0.$
By (\ref{2}), we have
\begin{equation}\label{4}
h(E(S))=\prod_{p} F_S\left(\frac{1}{p}\right).
\end{equation}
and
\begin{equation}\label{5}
\log h(E(S))=\sum_{p}\log F_S(x)|_{x=\frac{1}{p}}.
\end{equation}
Let
\begin{equation}\label{6}
\log F_S(x)=\sum_{i\geq2}\frac{f_i^{(S)}}{i}x^i.
\end{equation}
Since $|u(n)-u(n-1)|\leq1,$ then by (\ref{3}), $F_S(x)\leq1+\frac{x^2}{1-x}$ \enskip and
$0<\log F_S(x)\leq2x^2, \enskip x\in (0,\frac{1}{2}].$ Thus the series (\ref{5})
is absolutely convergent. Now,  according to (\ref{5}) - (\ref{6}), we have
\begin{equation}\label{7}
\log h(E(S))=\sum_{n=2}^{\infty}\frac{f_n^{(S)}}{n}P(n),
\end{equation}
where $P(n)=\sum_{p}\frac{1}{p^n}$ is the prime zeta function.
The series (\ref{7}) is fast convergent and very suitable for the calculation of $h(E(S)).$
\section{A recursion for coefficients}
Denoting
\begin{equation}\label{8}
v_n=u(n)-u(n-1),\enskip n\geq2,
\end{equation}
by (\ref{3}) and (\ref{6}),
we have
\begin{equation}\label{9}
F_S(x)=1+\sum_{n\geq2}v_nx^n,
\end{equation}
\begin{equation}\label{10}
\log(1+\sum_{n\geq2}v_nx^n)=\sum_{i\geq2}\frac{f_i^{(S)}}{i}x^i.
\end{equation}
\begin{lemma}\label{L1}
Coefficients $\{f_n^{(S)}\}$ satisfy the recurrence
\newpage
\begin{equation}\label{11}
f_{n+1}^{(S)}=(n+1)v_{n+1}-\sum_{i=1}^{n-2} v_{n-i}f_{i+1}^{(S)}, \enskip n\geq1.
\end{equation}
\end{lemma}

\begin{proof}
Differentiating (\ref{10}), we have
$$\frac{\sum_{n\geq2}nv_nx^{n-1}}{F_S(x)}=\sum_{j>=1}f_{j+1}^{(S)}x^j.$$
Hence,
$$\sum_{n\geq2}nv_nx^{n-1}=(1+\sum_{n\geq2}v_nx^{n})(\sum_{j>=1}
f_{j+1}^{(S)}x^j ).$$
Equating the coefficients of $x^n$ in both sides, we get
$$ (n+1)v_{n+1}=f_{n+1}^{(S)}+\sum_{j=1}^{n-2} v_{n-j}f_{j+1}^{(S)}$$
and the lemma follows.
\end{proof}
\begin{corollary}\label{C1}
All $\{f_{n}^{(S)}\}$ are integers.
\end{corollary}
\begin{proof}
For n=1,2,3, by the recurrence (\ref{11}), we have
$$f_{2}^{(S)}=2v_2, f_{3}^{(S)}=3v_3, f_{4}^{(S)}=4v_4-2v_2^2;$$
now the corollary follows by induction.
\end{proof}

\section{Explicit polynomial formula}
To apply (\ref{10}) we need a fast way to generate the coefficients $f_i^{(S)}.$
Since, for $x\in (0,\frac{1}{2}],$ $\sum_{n\geq2}v_nx^n\leq\frac{x^2}{1-x}
\leq\frac{1}{2},$ then
\begin{equation}\label{12}
\log(1+\sum_{n\geq2}v_nx^n)=\sum_{m\geq1}\frac{(-1)^{m-1}}{m}(\sum_{n\geq2}v_nx^n)^m.
\end{equation}
Expanding these powers, we get a great sum of terms of type
\begin{equation}\label{13}
t_{\lambda_1,s_1}(v_{\lambda_1}x^{\lambda_1})^{s_1}... t_{\lambda_r,s_r}(v_{\lambda_r}
x^{\lambda_r})^{s_r},\enskip s_i\geq1, \lambda_i\geq2.
\end{equation}

When we collect all the terms with a fixed sum of exponents of $x,$ say, $n,$
we get a sum of terms (\ref{13}) with $\lambda_1s_1+...+\lambda_rs_r=n,$ i.e.,
we have $s_i$ parts $\lambda_i$ in partition of $n.$ Therefore, the considered expansion has the form
$$\log(1+\sum_{n\geq2}v_nx^n)=\sum_{n\geq2}(\sum_{\sigma\in \Sigma_n}t_{\sigma}v_{\sigma})\frac{x^n}{n}=\sum_{n\geq2}\frac{f_{n}^{(S)}}{n}x^n,$$
where $\Sigma_n$ is the set of the partitions $\{\sigma\}$ of $n$ with parts
 $\lambda_i\geq2$
 and $t_{\sigma}, v_{\sigma}$ are functions of partitions $\sigma$ defined
 by (\ref{13}) such that with every partition $\sigma$ of $n$ we associate
 the monomial
 \newpage
 \begin{equation}\label{14}
 v_{\sigma}=\prod_{i=1}^{r}v_{\lambda_i}^{s_i} \enskip (\lambda_1s_1+...+
 \lambda_rs_r=n,\enskip \lambda_i\geq2).
 \end{equation}

 So
\begin{equation}\label{15}
f_{n}^{(S)}=\sum_{\sigma\in \Sigma_n}t_{\sigma}v_{\sigma}.
\end{equation}
Substituting (\ref{15}) in equation (\ref{11}), we get
$$\sum_{\sigma\in \Sigma_{n+1}}t_{\sigma}v_{\sigma}=(n+1)v_{n+1}-
\sum_{i=1}^{n-2}v_{n-i}\sum_{\sigma\in \Sigma_{i+1}}t_{\sigma}v_{\sigma}=$$
\begin{equation}\label{16}
(n+1)v_{n+1}-
\sum_{j=2}^{n-1}v_{j}\sum_{\sigma\in \Sigma_{n+1-j}}t_{\sigma}v_{\sigma}.
\end{equation}
Note that, using (\ref{16}), one can proved that all coefficients $t_{\sigma}$
are integer numbers.
Let partition $\sigma=(b_2,...,b_{n+1})\in \Sigma_{n+1}$ contains $b_2$ elements 2, ..., $b_{n+1}$
elements $n+1$ such that $2b_2+...+(n+1)b_{n+1}=n+1, \enskip b_i\geq0.$ In
particular, evidently, $b_{n+1}=0\enskip or \enskip 1$ and in the latter case
all other $b_i=0.$ We shall write
$v_{\sigma}=v_2^{b_2}...v_{n+1}^{b_{n+1}}$ and $t_{\sigma}=t(v_2^{b_2}...
v_{n+1}^{b_{n+1}}).$  According to (16), the coefficient of the monomial
$v_2^0...v_n^0v_{n+1}^1$ equals $n+1,$ i. e., for partition of $n+1$ with only
part we have $t(\sigma)=n+1.$ We agree that $0^0=1.$

Denote by $\Sigma'_{n+1}$ the set of partitions of $n+1$ with parts $\geq2$ and
$\leq n.$ Then, by (\ref{16}), we have
\begin{equation}\label{17}
\sum_{\sigma\in \Sigma'_{n+1}}t_{\sigma}v_{\sigma}=-\sum_{j=2}^{n-1}v_{j}\sum_{\sigma\in \Sigma'_{n+1-j}}t_{\sigma}v_{\sigma}.
\end{equation}
 For every partition
 $(b_2,...,b_{n+1})\in \Sigma'_{n+1}$
 we have $b_{n+1}=0$ and $b_n=0$ (the latter since all parts $\geq2).$ Then
 (\ref{17}) leads to the formula:
$$t(v_2^{b_2}...v_{n-1}^{b_{n-1}}v_n^0v_{n+1}^0)=-t(v_2^{b_2-1}v_3^{b_3}...
v_{n-1}^{b_{n-1}}v_n^0v_{n+1}^0)-$$
\begin{equation}\label{18}
t(v_2^{b_2}v_3^{b_3-1}
...v_{n-1}^{b_{n-1}}v_n^0v_{n+1}^0)-...-
t(v_2^{b_2}v_3^{b_3}...v_{n-1}^{b_{n-1}-1}v_n^0v_{n+1}^0).
\end{equation}
Using (\ref{18}), we find an explicit formula for $f_{n}^{(S)}.$
\begin{lemma}\label{L2}
Let, for $n\geq3,$ $(b_2,...,b_{n-1},0,0)\in \Sigma'_{n+1}.$ Then
\begin{equation}\label{19}
t(v_2^{b_2}...v_{n-1}^{b_{n-1}}v_n^0v_{n+1}^0)=(-1)^{B_{n-1}-1}\frac{(B_{n-1}-1)!}{b_2!...b_{n-1}!}(n+1),
\end{equation}
where $B_{n-1}=b_2+...+b_{n-1}.$
\end{lemma}

\begin{proof} Let $n=3.$ We saw that $f_{4}^{(S)}=4v_4-2v_2^2.$ So, $t(v_2^{b_2})=-2$
with
\newpage
 $b_2=2$ and, by (\ref{19}), we also obtain $t(v_2^{b_2})=-2.$ Let the lemma holds
for $t(v_2^{c_2}...v_{n-1}^{c_{n-1}}),\enskip n\geq3,$ where all $c_i\leq b_i$ such that not
all equalities hold. Then, by the relaion (\ref{18}) and the induction supposition,
we have
$$t(v_2^{b_2}...v_{n-1}^{b_{n-1}})=-(-1)^{B_{n-1}-2}(\frac{(B_{n-1}-2)!}{(b_2-1)!b_3!...
b_{n-1}!}(n+1-2)+$$ $$\frac{(B_{n-1}-2)!}{b_2!(b_3-1)!...b_{n-1}!}(n+1-3)+...+\frac{(B_{n-1}-1)!}{b_2!b_3!...
(b_{n-1}-1)!}(n+1-(n-1))= $$
$$ (-1)^{B_{n-1}-1}\frac{(B_{n-1}-2)!}{b_2!...b_{n-1}!}(b_2(n+1-2)+b_3(n+1-3)+...+$$
$$b_{n-1}(n+1-(n-1))=(-1)^{B_{n-1}-1}\frac{(B_{n-1}-2)!}{b_2!...b_{n-1}!}
(B_{n-1}(n+1)-$$ $$(2b_2+
3b_3+...+(n-1)b_{n-1})$$
and, since $2b_2+3b_3+...+(n-1)b_{n-1}=n+1,$ the lemma follows.
\end{proof}
\begin{corollary}\label{C2}
Let, for $n\geq3,$ $(b_2,...,b_{n+1})\in \Sigma_{n+1}.$ Then
\begin{equation}\label{20}
t(v_2^{b_2}...v_{n+1}^{b_{n+1}})=(\delta(b_{n+1,1})+(-1)^{B_{n-1}-1}
\frac{(B_{n-1}-1)!}{b_2!...b_{n-1}!})(n+1),
\end{equation}
where $B_{n+1}=b_2+...+b_{n-1}.$
\end{corollary}
\begin{proof}
The statement follows from Lemma \ref{L2} and addition of the coefficient $n+1$
of $v_{n+1}$ in equation (\ref{16}) in case when $\delta(b_{n+1,1})=1.$
\end{proof}
Now, using (\ref{7}), (\ref{15}), Corollary \ref{C2} and the initial values of the
coefficients $f_{2}^{(S)}=2v_2, \enskip f_{3}^{(S)}=3v_3,$ and changing $n$ by $n-1,$
we get a suitable formula to compute $\log h(E(S)).$
\begin{theorem}\label{t2}
We have
\begin{equation}\label{21}
\log h(E(S))=P(2)v_2+P(3)v_3+\sum_{n=4}^{\infty}P(n)(v_n+M(v_2,...,v_{n-2})),
\end{equation}
where $P(n)$ is the prime zeta function, $M$ is the polynomial defined as
$$M(v_2,...,v_{n-2})=\sum_{2b_2+...+(n-2)b_{n-2}=n}(-1)^{B_{n-2}-1}
\frac{(B_{n-2}-1)!}{b_2!...b_{n-2}!}
v_2^{b_2}...v_{n-2}^{b_{n-2}}, $$
where $B_{n-2}=b_2+...+b_{n-2}, \enskip b_i\geq0,\enskip i=2,...,n-2,
\enskip n\geq4.$
\end{theorem}
In particular, for $n=4,5,6,...,$ we have 
$$M(v_2)=-\frac{v_2^2}{2}, M(v_2,v_3)=-v_2v_3, M(v_2,v_3,v_4)=-v_2v_4-\frac{v_3^2}{2}+\frac{v_2^3}{3},\enskip...$$
For example, in case $n=6$ the diophantine equation $2b_2+3b_3+4b_4=6$ has 3
solutions \newpage
a) $b_2=1, b_3=0, b_4=1$ with $B_4=2;$\newline
\indent b) $b_2=0, b_3=2, b_4=0$ with $B_4=2;$\newline
\indent c) $b_2=3, b_3=0, b_4=0$ with $B_4=3.$\newline

Besides, using (\ref{11}), for $M_n=M_n(v_2,...,v_{n-2})$ we have the recursion
 \begin{equation}\label{22} 
M_2=0, M_3=0, M_n=-\frac{1}{n}\sum_{j=2}^{n-2}jv_{n-j}(v_j+M_j), \enskip n\geq4
\end{equation}
which, possibly, more suitable for fast calculations by Theorem \ref{t2}.
\section{Examples}
1) As we already mentioned, in case when $S$ is the sequence of square-free numbers,
Arias de Reyna [5,A262276] obtained
$$ h= \prod_{p}\left(1+\sum_{i\geq4}\frac{\mu(i)^2-\mu(i-1)^2}{p^i}\right)
=0.95592301586190237688...$$
By the results of \cite{1}, the coefficients $f_n^{(S)}$ (\ref{15}) in this case
(see A262400 \cite{5}) have very interesting congruence properties.\newline
\indent 2) The case of $S=2^n$ was essentially considered by the author \cite{3}. He found
 that $h=0.872497...$ The author asked Arias de Reyna to get more digits.
 Using Theorem \ref{t2}, he obtained
 $$h=0.87249717935391281355...$$
 \indent 3) Among the other several calculations by Arias de Reyna, we give the following
 one. Let $S$ be 1 and the primes (A008578 \cite{5}). Then
 $$h=0.94671933735527801046...  $$

 \section{Acknowledgement}
 The author is very grateful to Juan Arias de Reyna for an information of Wrench's
 method, useful discussions and his calculations by the formula of Theorem \ref{t2}.

\end{document}